\documentclass[11pt]{amsart}

\usepackage{amssymb,amsmath,amsfonts,amstext,latexsym}
\usepackage[all]{xy}
\usepackage[colorlinks=true, urlcolor=rltblue, citecolor=drkgreen, linkcolor=drkred] {hyperref}
\usepackage{color}
\usepackage{pdfsync}
\usepackage{enumitem}
\definecolor{rltblue}{rgb}{0,0,0.4}
\definecolor{drkgreen}{rgb}{0,0.4,0}
\definecolor{drkred}{rgb}{0.5,0,0}

\newtheorem{thm}{Theorem}
\newtheorem{lemma}[thm]{Lemma}
\newtheorem{proposition}[thm]{Proposition}

\newtheorem{theorem}[thm]{Theorem}
\newtheorem{corollary}[thm]{Corollary}

\theoremstyle{definition}
\newtheorem{definition}[thm]{Definition}

\theoremstyle{remark}

\newtheorem{remark}[thm]{Remark}

\newtheorem{historic}[thm]{Historic Remark}
\newtheorem{claim}{Claim}

\theoremstyle{plain}

\newcounter{contenumi}

\def\upto{\mathop{\upharpoonright}}

\def\and{\mathrel{\&}}

\def\isom{\cong}
\def\Si{\Sigma}

\def\A{\mathcal{A}}

\def\B{\mathcal{B}}

\def\I{\mathcal{I}}

\def\om{\omega}
\def\bbar{\bar{b}}

\def\a{\alpha}

\def\A{\mathcal A}

\def\L{\mathcal L}

\def\cbar{{\bar{c}}}

\def\bbar{{\bar{b}}}

\def\dbar{{\bar{d}}}
\def\ctt{{\mathtt c}}
\def\itt{{\mathtt {in}}}
\def\Sic{\Si^\ctt}
\def\Dec{\Delta^\ctt}
\def\Pic{\Pi^\ctt}

\def\Sii{\Si^\itt}

\def\Dei{\Delta^\itt}

\def\bfDelta{{\bf \Delta}}

\def\Ahat{\hat{\A}}

\def\Si{\Sigma}

\def\om{\omega}

\def\forces{\Vdash}
\def\nforces{\nVdash}

\def\wck{\om_1^{CK}}
\def\omCK{\wck}

\def\wck{\om_1^{CK}}

 \def\Ahat{{\hat{\A}}}
 \def\Bhat{{\hat{\B}}}

\newcommand{\J}{\mathcal{J}}

\def\Ati{{\widetilde{\A}}}
\def\Bti{{\widetilde{\B}}}

\def\Do{\mathcal Dom}

\newcounter{been}

\newcounter{eebeen}

\newcommand{\Iso}[1]{\text{Iso}(#1)}

\renewcommand{\Ahat}{{\widehat{\mathcal{A}}}}
\renewcommand{\Bhat}{{\widehat{\mathcal{B}}}}

\newcommand{\Atilde}{{\widetilde{\mathcal{A}}}}
\newcommand{\Btilde}{{\widetilde{\mathcal{B}}}}

\newcommand{\comment}[1]{}
\newcommand{\res}{\upto}
\DeclareMathOperator{\id}{id}
\DeclareMathOperator{\negat}{neg}
\newcommand{\Dom}[2]{\Do_{#2}^{#1}}

\def\Fra{\mathfrak{F}}
\def\Gra{\mathfrak{G}}

\def\Aut{{\textrm{Aut}}}

\title{Borel Functors and Infinitary Interpretations}

\author[M. Harrison-Trainor]{Matthew Harrison-Trainor}
\address{Group in Logic and the Methodology of Science\\
University of California, Berkeley\\
 USA}
\email{matthew.h-t@math.berkeley.edu}
\urladdr{\href{http://math.berkeley.edu/~mattht}{http://math.berkeley.edu/$\sim$mattht}}

\author[R. Miller]{Russell Miller}
\address{Mathematics Dept., Queens College; Ph.D.\ Programs in Mathematics \& Computer Science, Graduate Center\\
City University of New York\\
USA}
\email{Russell.Miller@qc.cuny.edu}
\urladdr{\href{http://qcpages.qc.cuny.edu/~rmiller}{http://qcpages.qc.cuny.edu/$\sim$rmiller}}

\author[A. Montalb\'an]{Antonio Montalb\'an}
\address{Department of Mathematics\\
University of California, Berkeley\\
	USA}
\email{antonio@math.berkeley.edu}
\urladdr{\href{http://www.math.berkeley.edu/~antonio/index.html}{www.math.berkeley.edu/$\sim$antonio}}

\thanks{The first author was partially supported by the Berkeley Fellowship and NSERC grant PGSD3-454386-2014.
The second author was supported by NSF grant \# DMS-1362206
and by several PSC-CUNY research awards.
The third author was partially supported by the Packard Fellowship and NSF grant \# DMS-1363310.
This work took place in part at a workshop held by the Institute for Mathematical Sciences of the National University of Singapore.}

\begin{document}

\begin{abstract}
We introduce the notion of infinitary interpretation of structures.
In general, an interpretation between structures induces a
continuous homomorphism between their automorphism groups, and furthermore,
it induces a functor between the categories of copies of each structure.
We show that for the case of infinitary interpretation the reversals are also true:
Every Baire-measurable  homomorphism between the automorphism groups of two countable structures is induced by an infinitary interpretation, and
every Baire-measurable functor between the set of copies of two countable structures is induced by an infinitary interpretation.
Furthermore, we show the complexities are maintained in the sense that if the functor is $\bfDelta^0_\alpha$, then the interpretation that induces it is $\Dei_\alpha$ up to $\bfDelta^0_\alpha$ equivalence. 
\end{abstract}

\maketitle


%

\section{Introduction}

Constructions that build new structures out of old ones are common throughout mathematics. For instance, given an integral domain $\B$, we might consider its fraction field or its polynomial ring. In model theory, a common way  of performing such constructions is using {\em interpretations}, where one structure is defined using tuples from the other, and the operations and relations of the new structure are defined using the operations and relations of the old one.
For instance, the fraction field of an integral domain $\B$ can be defined as a set of pairs of elements in $\B$ quotiented out by some definable equivalence relation, with the operations on the pairs defined using the operations in $\B$.
Interpretations are useful because they preserve some model theoretic properties of the structures or of their theories. For instance, if a structure $\A$ can be interpreted within a structure $\B$, then there is a homomorphism from the automorphism group of $\B$ to the automorphism group of $\A$ (\cite[Theorem 5.3.5]{Hodges93}).
Furthermore, if we assume these structures are countable, we get a function that maps copies of $\B$ with domain $\om$ to copies of $\A$ with domain $\om$, and one that maps isomorphisms between copies of $\B$ to isomorphisms between the respective copies of $\A$, preserving compositions.
We naturally view such a pair of functions as a {\em functor} from $\B$ to $\A$ (Definition \ref{def:functors}).
Functors induced by interpretations are always Borel.
In turn, a Borel functor (even a Baire-measurable one) from $\B$ to $\A$ induces a continuous homomorphism from the automorphism group of $\B$ to that of $\A$.
However, there are many such functors that do not come from elementary first-order interpretations.\footnote{For example, there is the functor which maps every copy of the trivial structure $\B$, with a countable domain and no relations, to the single structure $\A=(\omega,0,1,+,\cdot)$, and maps all isomorphisms  between copies of $\B$ to the identity map on $\A$.}
In the case of the polynomial ring, we can easily build a functor that maps copies of a ring $\B$ to copies of its polynomial ring $\B[X]$ in such a way that isomorphisms between copies of $\B$ translate to isomorphisms between the respective copies of $\B[X]$. 
However, $\B[X]$ cannot be interpreted in $\B$,
as we need tuples of arbitrary large size from $\B$ to code the polynomials in $\B[X]$.
In this paper we consider a more general notion of interpretation that we call {\em infinitary interpretation}, where the sets used in the interpretation need only be $\L_{\om_1 \om}$-definable and where, instead of using tuples of a fixed size for the interpretation, we allow tuples of different sizes (see Definition \ref{def: inf int} below).
We only consider countable structures, and so, whenever we refer to a structure, we assume it is countable and with domain $\omega$. 
These new interpretations still generate Borel functors from the interpreting structure to the interpreted structure exactly as above, and also continuous homomorphisms between their automorphism groups.
Our main theorem is the reversal: Each Borel (even Baire measurable) functor  from the copies of $\B$ to the copies of $\A$ is naturally isomorphic to one induced by an infinitary interpretation (Theorem \ref{thm:main}), and each continuous homomorphism $\Aut(\B)\to\Aut(\A)$ is induced by an infinitary interpretation (Theorem \ref{homo-to-interp}). 
Furthermore, the quantifier complexity of the interpretation is the same as the Borel complexity of the functor.
In a sense, this shows that infinitary interpretations are the most general kind of interpretations, at least if we restrict ourselves to countable structures.
One can view our result as saying that if one has a way of building $\A$ from $\B$, then a copy of $\A$ must already exist inside of $\B$.

Continuing this line of investigation, we obtain results towards the following question: What can we tell about a structure by looking at its automorphism group?
The first question along this lines that we consider is whether there is a syntactical condition on structures that is equivalent to them having the same automorphism group.  The answer is {\em infinitary bi-interpretability}: Two structures $\mathcal{A}$ and $\mathcal{B}$ are infinitarily bi-interpretable if each can be infinitarily interpreted in the other and the isomorphism taking $\mathcal{A}$ to the copy of $\mathcal{A}$ inside the copy of $\mathcal{B}$ inside $\mathcal{A}$, and the similar isomorphism with $\mathcal{A}$ and $\mathcal{B}$ reversed, are infinitarily definable in the respective structures (Definition \ref{defn:inf-biinterpretable}). 
This is equivalent to the existence of a continuous isomorphism between the
automorphism groups of the structures (Theorem \ref{thm:inf-bi-iso}).
(For the particular case of $\aleph_0$-categorical structures, this was already known
from a paper of Ahlbrandt and Ziegler \cite{AhlbrandtZiegler}.)
To show this, we prove that infinitary bi-interpretations correspond naturally (and bijectively) to Borel adjoint equivalences (Definition \ref{def: adj equiv cat}) of the categories of copies of the structures (Theorems \ref{thm:bi to eq} and \ref{thm:eq to bi}).
The second question is whether, and how, the existence of a set of indiscernibles within a structure is reflected in the automorphism group of the structure.
We will show that a structure has an infinitarily definable set of indiscernible equivalence classes on its tuples if and only if there is a continuous homomorphism from its automorphism group onto $S_\infty$  (Theorem \ref{thm:indiscernibles}).

This work grew out of a previous paper \cite{HTM3} by the three authors and Alexander Melnikov, which gives a one-to-one correspondence between effective interpretations and effective functors.
Effective interpretation is the right notion of interpretability needed for computability theory, and is exactly like  infinitary interpretation as we define it in Definition \ref{def: inf int}, but using only computable infinitary $\Si_1$-formulas.
This particular definition was introduced in \cite{MonFixed, MonICM}, but it is equivalent to the notion of $\Sigma$-definability without parameters, widely studied in Russia. 
On the other hand, the precise definition of computable functor was introduced in \cite{MPSS}, where it was shown to show all structures can be effectively coded by fields.
Both effective bi-interpretations and computable functors were introduced to formalize a longstanding idea from \cite{HKSS} that certain classes of structures are universal for computability-theoretic properties.
Some time later, we realized that with some more work, and via the use of forcing, we could extend our results from \cite{HTM3} through the Borel hierarchy.
We then noticed we could apply our results to homomorphisms between automorphism groups and infinitary indiscernibles.

\subsection{Infinitary interpretations}
\label{subsec:infint}

Let us now formally define the notion of infinitary interpretation.
Throughout this article, all signatures are relational and computable:
there is a computable function giving the arity of each of the countably
many predicates $P_0,P_1,\ldots$.  (It seems fairly clear that one could
extend our arguments to noncomputable countable signatures by
relativizing everything to the Turing degrees of the signatures.)

\begin{definition}\label{def: inf int}
A structure $\A = (A; P_0^\A,P_1^\A,...)$ (where $P_i^\A\subseteq A^{a(i)}$) is {\em infinitarily interpretable} in $\B$ if there are relations $\Dom{\A}{\B}$, $\sim$, $R_0, R_1,..$,
each $L_{\om_1,\om}$-definable without parameters in the language of $\B$,  such that
\begin{enumerate}
	\item $\Dom{\B}{\A}\subseteq \B^{<\om}$,
	\item $\sim$  is an equivalence relation on $\Dom{\B}{\A}$, 
	\item $R_i\subseteq (\Dom{\B}{\A})^{a(i)}$ is closed under $\sim$,
\end{enumerate}
and there exists a function $f^\B_\A\colon \Dom{\B}{\A} \to \A$ which induces an isomorphism: 
\[
f^\B_\A\colon (\Dom{\B}{\A}/\sim; R_0/\sim,R_1/\sim,...) \isom  (A; P_0^\A,P_1^\A,...) ,
\]
where $R_i/\!\sim$ stands for the $\sim$-collapse of $R_i$.
\end{definition}

In the definition above, when we refer to an $\L_{\om_1,\om}$-definable subset $S\subseteq \B^{<\om}$ we mean a countable sequence $\{S_1,S_2,\ldots\}$ of $\L_{\om_1,\om}$-definable subsets $S_i\subseteq \B^{i}$.
We refer the reader to \cite[Chapters 6 and 7]{HolyBible} for background on the infinitary language and its effective version.  Our notation $\Sic_\alpha$ refers to computable infinitary
$\Sigma_\alpha$ formulas (necessarily with $\alpha<\omCK$), and likewise for $\Dec_\alpha$.
We also sometimes use $\Sii_\alpha$, simply to emphasize that infinitary
(not necessarily computable) formulas are included.

We only deal with countable structures in this paper, and for a relation on a countable structure, being $\mathcal{L}_{\om_1\om}$ definable is equivalent to being invariant under automorphisms (\cite{Kueker68,Makkai69}).
One might then say that this is not really a syntactical definition. 
However, we will also be interested in the complexity of the interpretations defined in terms of the syntactic complexity of the formulas.
We say that an interpretation is $\Dei_\alpha$, or $\Dec_\alpha$, if all the relations $\Dom{\A}{\B}, \sim, R_0, R_1,..$ are.
(In the lightface case, when we refer to a $\Dec_\alpha$-definable subset $S\subseteq \B^{<\om}$ we mean a computable sequence (of indices) $\{s_1,s_2,...\}$ of $\Dec_\alpha$-definable subsets $S_i\subseteq \B^{i}$.
Similarly, when we refer to a $\Dec_\alpha$-definable sequence $\{S_k:k\in\om\}$ of subsets of  $\B^{<\om}$, we mean that the sequence of indices for the $S_k$'s is computable.)

Notice that given a presentation of a structure $\A$ with domain $\om$, we get a presentation of $\Aut(\A)$ as a subgroup of $S_\infty$.
The automorphism group of a different presentation would be a different subgroup of $S_\infty$, although these two subgroups would be conjugated by the isomorphism between the presentations.
Given fixed copies of $\A$ and $\B$ with domain $\om$, an infinitary interpretation induces a map between their automorphism groups in an obvious way.

\begin{definition}
\label{defn:G_I}
To each interpretation $\I$ of $\A$ in $\B$ as in Definition \ref{def: inf int}, we associate a
homomorphism $G_\I\colon \Aut(\B)\to\Aut(\A)$ as follows:
$$G_\I(f) =  f^\B_\A \circ \tilde{f} \circ ({f^\B_\A})^{-1}.$$
Here $\tilde{f}$ permutes $\Dom{\B}{\A}$ as defined by the given $f$, and hence preserves $\sim$.
\end{definition} 

Throughout this paper, we use $\tilde{f}$ (where $f$ is a map with domain $\B$) to denote the induced map on tuples from $\Dom{\B}{\A}$.

It is not hard to see that $G_\I$ is a continuous homomorphism. 
(Let us remark that every Baire-measurable homomorphism between Polish groups is continuous (see \cite[Theorem 2.3.3]{GaoBook}) and all automorphism groups are Polish (see \cite[Exercise 2.4.7]{GaoBook}).)
One of the main results of this paper is that all continuous homomorphisms between automorphism groups are induced by infinitary interpretations. 

\begin{theorem}\label{homo-to-interp}
Let $\A$ and $\B$ be countable structures.
Every continuous homomorphism from $\Aut(\B)$ into $\Aut(\A)$ is of the form $G_\I$ for some infinitary interpretation $\I$ of $\A$ in $\B$.
\end{theorem}

Note that there do exist structures whose automorphism groups are isomorphic as groups, but not as topological groups \cite{EvansHewitt}. So we cannot drop the hypothesis of continuity. On the other hand, there are models of $\text{ZF} + \text{DC}$ such that every homomorphism between polish groups is continuous \cite{Solovay,Shelah} and so we might expect such examples to be the exception.
 
As a corollary of this theorem we give a characterization, in terms of the automorphism group of a structure $\A$, for $\A$ to have an absolutely indiscernible set of $\L_{\omega_1 \omega}$-imaginary elements (i.e., an absolutely indiscernible set of equivalence classes under some $\L_{\omega_1 \omega}$-definable equivalence relation).

\begin{theorem}\label{thm:indiscernibles}
Let $\A$ be a countable structure. 
The following are equivalent:
\begin{enumerate}
\item There is a continuous homomorphism from $\Aut(\A)$ onto $S_\infty$.
\item There is an $n$, an $\L_{\om_1 \om}$-definable $D\subset A^n$, and an $\L_{\om_1 \om}$-definable equivalence relation $E\subset D^2$ with infinitely many equivalence classes and such that the $E$-equivalence classes are absolutely indiscernible, in the sense that every permutation of the $E$-equivalence classes extends to an automorphism of $\A$.
\end{enumerate}
\end{theorem}

The theorem above shows the connections behind the new proof by Baldwin, Friedman, Koerwein, and Laskowski \cite{BFKL} and the original proof of a result of Hjorth \cite{Hjorth07} that states that if there is a counterexample to Vaught's conjecture, there is one with no copies of size $\aleph_2$.
For this, Hjorth's proof started by considering a structure whose automorphism group divides $S_\infty$ (i.e., there is an onto continuous homomorphism from a closed subgroup of the automorphism group onto $S_\infty$) and then used descriptive set theoretic tools. 
This proof is hard to visualize for those outside of descriptive set theory,
and so Baldwin, Friedman, Koerwein, and Laskowski found another proof starting from a structure that has a set of absolute indiscernibles.
It is suggested in  \cite{BFKL} that the use of absolute indiscernibles is in a sense the model theoretic version of the use of the divisibility of $S_\infty$ by the automorphism group.
The theorem above makes this sense precise. 

In general it is necessary that we look at equivalence classes to find the indiscernibles, as it was shown in \cite{HIK} that every structure is bi-interpretable with one that has no triple of indiscernibles.

We also show (Theorem \ref{thm: order indiscernibles}) that a structure has absolute order indiscernibles if and only if there is a continuous homomorphism from $\Aut(\A)$ onto $\Aut(\mathbb{Q})$.

We will also consider bi-interpretations. Two structures are bi-interpretable if they are each interpretable in the other, and the compositions are definable:

\begin{definition}
\label{defn:inf-biinterpretable}
Two structures $\A$ and $\B$ are {\em infinitarily bi-interpretable} if there are interpretations of each structure in the other as in Definition \ref{def: inf int} such that the compositions
\[
f^\A_\B \circ \tilde{f}^\B_\A\colon \Do_\B^{(\Do_\A^\B)} \to \B 
\quad\mbox{ and }\quad
f^\B_\A \circ \tilde{f}^\A_\B\colon \Do_\A^{(\Do_\B^\A)} \to \A 
\]
are $\L_{\omega_1 \omega}$-definable in $\B$ and $\A$ respectively.
(Here $\Do_\B^{(\Do_\A^\B)}\subseteq (\Do_\A^\B)^{<\om}$, and $\tilde{f}^\B_\A\colon (\Do_\A^\B)^{<\om}\to \A^{<\om}$ is the obvious extension of $f^\B_\A\colon \Do_\A^\B\to \A$ mapping $\Do_\B^{(\Do_\A^\B)}$ to $\Do_\B^\A$.)
\end{definition}

Two structures which are bi-interpretable behave in the same way. In particular, we get a continuous isomorphism of the automorphism groups of the two structures. For this, the fact that the two $\L_{\omega_1 \omega}$-definable isomorphisms are of the form $f^\A_\B \circ \tilde{f}^\B_\A$ and $f^\B_\A \circ \tilde{f}^\A_\B$ for some $f^\B_\A$ and $f^\A_\B$ is vital.

\begin{theorem}\label{thm:inf-bi-iso}
Two countable structures $\A$ and $\B$ are infinitarily bi-interpret\-able if and only if their automorphism groups are Baire-measurably isomorphic. 
Furthermore,  every continuous isomorphism from $\Aut(\B)$ onto $\Aut(\A)$ is of the form $G_\I$ for some infinitary bi-interpretation $\I$ of $\A$ in $\B$.
\end{theorem}

\subsection{Functors}

Throughout the paper, 
we write $\Iso{\A}$ for the isomorphism class of a countably infinite structure $\A$:
\[
 \Iso{\A} = \{\Ahat~:~\Ahat\cong\A~\&~\text{dom}(\Ahat)=\omega\}.
 \]
We will regard $\Iso{\A}$ as a category, with the copies of the structure as its objects and the isomorphisms among them as its morphisms.

\begin{definition}\label{def:functors}
By a {\em functor from $\A$ to $\B$} we mean a functor from $\Iso{\A}$ to  $\Iso{\B}$, that is, a map $F$ that assigns to each copy $\Ahat$ in $\Iso{A}$ a structure $F(\Ahat)$ in $\Iso{\B}$, and assigns to each morphism $f\colon \Ahat \to \Atilde$ in $\Iso{\A}$ a morphism $F(f) \colon F(\Ahat) \to F(\Atilde)$ in $\Iso{\B}$ so that the  two properties below hold:
	\begin{enumerate}
		\item[(N1)]	$F(\id_{\Ahat}) = \id_{F(\Ahat)}$ for every $\Ahat \in \Iso{\A}$, and
		\item[(N2)] $F(f \circ g) = F(f) \circ F(g)$ for all morphisms $f,g$ in $\Iso{\A}$.
	\end{enumerate}
$F$ is $\Delta^0_\alpha$ (or $\bfDelta^0_\alpha$) if it is given by a pair of $\Delta^0_\alpha$ (resp. $\bfDelta^0_\alpha$) operators $2^\omega \to 2^\omega$. It is Borel if it is given by Borel operators, and Baire-measurable if it is given by Baire-measurable operators.
\end{definition}

Every interpretation $\I$ of a structure $\A$ in a structure $\B$ induces an functor, $F_{\I}$, from $\B$ to $\A$.
There is only one small technicality in the definition of $F_{\I}$, which has to do with making the domain of $F_{\I}(\Bhat)$ equal to $\om$.
Using the interpretation we can associate, to each copy $\Bhat$ of $\B$, a copy of $\A$ whose domain consists of the $\sim$-equivalence classes of $\Dom{\Bhat}{\A}\subseteq \om^{< \om}$;
Using a bijection $\tau^{\Bhat}$ between $\om$ and $\Dom{\Bhat}{\A} / \sim$ (defined in some canonical way using an effective bijection between $\om$ and $\om^{<\om}$, so that
we can compute $\tau^{\Bhat}$ from $\Dom{\Bhat}{\A}$ and $\sim$), we then define $F_{\I}(\Bhat)$ to be the pull-back of this structure through $\tau^{\Bhat}$. If $h$ is an isomorphism $\Bhat \to \Btilde$, we define $F_{\I}(h) \colon F_{\I}(\Bhat) \to F_{\I}(\Btilde)$ by $F_{\I}(h) = (\tau^{\Btilde})^{-1} \circ \tilde{h} \circ \tau^{\Bhat}$.

Our main theorem states that every functor from $\B$ to $\A$ is of the form $F_{\I}$ up to natural isomorphism.

\begin{definition}\label{def:effeq}
A functor $F\colon \Iso{\B} \rightarrow \Iso{\A}$ is {\em naturally isomorphic} (or just {\em isomorphic}) to a functor $G\colon \Iso{\B} \rightarrow \Iso{\A}$ if for every $\Btilde \in \Iso{\B}$, there is an isomorphism $\eta_{\Btilde} \colon F(\Btilde) \to G(\Btilde)$, such that the following diagram commutes for every $\Btilde, \Bhat \in \Iso{\B}$ and every morphism $h\colon \Btilde \to \Bhat$:
\[
\xymatrix{
F(\Btilde)\ar[d]_{F(h)}\ar[r]^{\eta_{\Btilde}} &      G(\Btilde)\ar[d]^{G(h)}   \\
F(\Bhat)\ar[r]_{\eta_{\Bhat}}    & G(\Bhat)
}\]
An isomorphism is Borel (or $\Delta^0_\alpha$, or $\bfDelta^0_\alpha$) if $\eta$ is given by a Borel (resp. $\Delta^0_\alpha$ or $\bfDelta^0_\alpha$) operator.
\end{definition}

The following is the key result of the paper and Section \ref{sec: construction} is dedicated to proving it.

\begin{theorem}\label{thm:main}
Let $\B$ and $\A$ be countable structures, possibly in different countable languages.
For each Baire-measurable functor $F\colon \Iso{\B} \to \Iso{\A}$ there is an infinitary interpretation $\I$ of $\A$ within $\B$, such that $F$ is naturally isomorphic to the functor $F_\I$ associated to $\I$.
Furthermore, if $F$ is $\Delta^0_\alpha$ in the lightface Borel hierarchy, then the interpretation can be taken to be $\Dec_\alpha$ and the isomorphism between $F$ and $F_\I$ can be taken to be $\Delta^0_\alpha$.
\end{theorem}

We also get a similar way of moving between bi-interpretations and a category-theoretic equivalent. For bi-interpretations, we must consider adjoint equivalences of categories.

\begin{definition}\label{def: adj equiv cat}
An adjoint equivalence of categories consists of two functors, from one category to the other and back, such that their compositions are both naturally isomorphic to the identity functors, and furthermore, these two natural isomorphisms are mapped to each other via these two functors. 
More formally, functors $F \colon \Iso{\B} \to \Iso{\A}$ and $G \colon \Iso{\A} \to \Iso{B}$, together with families of isomorphisms $\epsilon_{\Atilde} \colon \Atilde \to F(G(\Atilde))$ and $\eta_{\Btilde} \colon \Btilde \to G(F(\Btilde))$ for $\Atilde\in \Iso{\A}$ and $\Btilde\in\Iso{\B}$, form an {\em adjoint equivalence of categories} if
\[ F(\eta_{\Bhat}) = \epsilon_{F(\Bhat)} \text{ and } G(\epsilon_{\Ahat}) = \eta_{G(\Ahat)}.\]
An adjoint equivalence of categories is Borel if $F$, $G$, $\eta$, and $\epsilon$ are Borel operators.
\end{definition}

For bi-interpretations, both directions---producing an equivalence of categories from a bi-interpretation, and vice versa---are non-trivial.

\begin{theorem}\label{thm:bi to eq}
Let $\B$ and $\A$ be countable structures.
For every infinitary bi-interpretation $(\I,\J)$ of $\A$ and $\B$, $F_\I$ and $F_\J$ form a Borel adjoint equivalence of categories of $\Iso{\B}$ and $\Iso{\A}$. Furthermore, complexities are maintained.
\end{theorem}

\begin{theorem}\label{thm:eq to bi}
Let $\B$ and $\A$ be countable structures.
For every Borel adjoint equivalence of categories $(F,G)$ between $\Iso{\B}$ and $\Iso{\A}$ there is an infinitary bi-interpretation $(\I,\J)$ between $\A$ and $\B$, such that $F$ and $G$ are naturally isomorphic to the functors $F_\I$ and $F_\J$ associated to $\I$ and $\J$ respectively. Furthermore, complexities are maintained.
\end{theorem}

\section{Homomorphisms of automorphism groups}

Our main result, Theorem \ref{thm:main}, shows the connection between functors and interpretations.
In this section, we discuss the connection between homomorphisms of automorphism groups and functors, which we will then be able to connect to interpretations once we prove Theorem \ref{thm:main}.

\begin{theorem}\label{thm:homo-to-functor}
For every continuous homomorphism $H\colon \Aut(\B)\to\Aut(\A)$, there is a Borel functor $G\colon \Iso{\B} \to \Iso{\A}$ with $G(\B) = \A$ and whose restriction to $\Aut(\B)$ is $H$.
\end{theorem}
\begin{proof}
Let $\Gamma$ be a map that assigns, to each copy $\Bhat$ of $\B$, an isomorphism $\Gamma^\Bhat\colon \Bhat\to\B$ with $\Gamma^{\B} = \id_{\B}$.
Let us first show how will use $\Gamma$, and then show how we can choose it to be Borel.

Using $\Gamma$ and $H$ we define $G$ as follows.
First, the action of the functor on the copies of $\B$ is trivial: For every copy $\Bhat$ of $\B$, we let $G(\Bhat)= \A$.
The action of $G$ on the isomorphisms is a bit more interesting: If $f\colon \Bhat\to\Bti$ is an isomorphism, then ${\Gamma^{\Bti}} \circ f \circ {\Gamma^\Bhat}^{-1}$ is an automorphism of $\B$, and we can define $G(f) = H({\Gamma^{\Bti}} \circ f \circ {\Gamma^\Bhat}^{-1})\in\Aut(\A)$.
It is not hard to check that $G$ is a functor. If $f \in \Aut(\B)$, then since $\Gamma^{\B} = \id_{\B}$, $G(f) = H(f)$.  Moreover, the continuity of $H$ and the fact that $\Gamma$ is Borel ensure that $G$ is Borel.

Let us now build $\Gamma$ in a Borel way.
Let $\alpha$ be the Scott rank of $\Bhat$ in the sense of \cite{MonScott}.
So, by \cite[Theorem 1.1]{MonScott}, $\B$ is uniformly $\Delta^0_\a$-relatively categorical on a cone, say the cone above $X$ .
Let $\Gamma$ be the operator witnessing this uniformity.
Note that we can choose $\Gamma^{\B} = \id_{\B}$.
\end{proof}

\begin{corollary}\label{cor:baire-to-borel}
Every Baire-measurable functor $F\colon \Iso{\B} \to \Iso{\A}$ is naturally isomorphic to a Borel one.
\end{corollary}
\begin{proof}
Fix the presentations of $\A$ and $\B$, with $\A=F(\B)$.
When restricted to the automorphisms of $\B$, $F$ is a Baire-measurable homomorphism from $\Aut(\B)$ to $\Aut(\A)$.
As we mentioned earlier, such a homomorphism must be continuous (see \cite[Theorem 2.3.3]{GaoBook}).
We can then apply the previous theorem to get a Borel functor $G\colon \Iso{\B} \to \Iso{\A}$ which coincides with $F$ on $\Aut(\B)$.
Write $H=F\res\Aut(\B)$ as above.
Then $F$ and $G$ are isomorphic:
Let $\Gamma$ be a map that assigns, to each copy $\Bhat$ of $\B$, an isomorphism $\Gamma^\Bhat\colon \Bhat\to\B$, as in the previous theorem.
Given a copy $\Bhat$ of $\B$, let $\eta_{\Bhat} = F(\Gamma^{\Bhat}) \colon F(\Bhat)\to G(\Bhat)$ (recall that $G(\Bhat)=\A = F(\B)$). We claim that $\eta$ is a natural isomorphism between $F$ and $G$.
Given an isomorphism $f\colon \Bhat\to\Bti$, and using the fact that $G\upto \Aut(\B) = F\upto \Aut(\B)=H$, we have
\begin{align*}
G(f) \circ \eta_\Bhat &= H({\Gamma^{\Bti}} \circ f \circ {\Gamma^\Bhat}^{-1}) \circ F(\Gamma^{\Bhat}) \\
&= F({\Gamma^{\Bti}} \circ f \circ {\Gamma^\Bhat}^{-1}) \circ F(\Gamma^{\Bhat}) \\
&= F({\Gamma^{\Bti}}) \circ F(f)\\
&= \eta_{\Bti}\circ F(f).\qedhere
\end{align*}
\end{proof}

Theorem \ref{thm:homo-to-functor} together with our main Theorem \ref{thm:main} provides a proof of Theorem \ref{homo-to-interp}, that each homomorphism between automorphism groups is induced by an infinitary interpretation. 
We can then use this to define a measure of complexity for homomorphisms between automorphism groups.

\begin{definition}
Given a continuous homomorphism $H\colon \Aut(\B)\to\Aut(\A)$, we define the {\em rank} of $H$ to be the least $\a$ such that there is a $\bfDelta^0_\alpha$-functor from $\B$ to $\A$ coinciding with $H$ on $\Aut(\B)$, or equivalently, a $\Dei_\alpha$ interpretation $\I$ of $\A$ within $\B$ with $H=G_\I$ as in Definition \ref{defn:G_I}.
From the proof of Theorem \ref{thm:homo-to-functor} we get that the rank of $H$ is at most the Scott rank of $\B$.
\end{definition}

Note that the rank of a homomorphism depends on the underlying structures $\A$ and $\B$, and not just on their automorphism groups.
We will not develop this notion of rank any further in this paper, but it seems so natural that we think it deserves further study.

We now turn to the connection between isomorphisms of automorphism groups and adjoint equivalences of categories.

\begin{theorem}\label{thm:eq-to-iso}
Let $F \colon \Iso{\B} \to \Iso{\A}$, $G \colon \Iso{\A} \to \Iso{\B}$, $\eta$, and $\epsilon$ form a Borel adjoint equivalence of categories between $\Iso{\A}$ and $\Iso{\B}$ with $F(\B) = \A$. Then $F$, restricted to $\Aut(\B)$, gives an isomorphism between $\Aut(\B)$ and $\Aut(\A)$.
\end{theorem}
\begin{proof}
Let $H_1 \colon \Aut(\B) \to \Aut(\A)$ be defined by $H_1(h) = F(h)$, and let $H_2 \colon \Aut(\A) \to \Aut(\B)$ be defined by $H_2(g) = \eta_{\B}^{-1} \circ G(g) \circ \eta_{\B}$. Then
\[ H_1 \circ H_2(h) = F(\eta_{\B}^{-1}) \circ F(G(h)) \circ F(\eta_{\B}) = \epsilon_{\A}^{-1} \circ F(G (h)) \circ \epsilon_{\A} = h \]
and
\[ H_2 \circ H_1(g) = \eta_{\B}^{-1} \circ G(F(g)) \circ \eta_{\B} = g. \qedhere \]
\end{proof}

\begin{theorem}\label{thm:iso-to-functor}
For every continuous isomorphism $H\colon \Aut(\B)\to\Aut(\A)$, there is a Borel adjoint equivalence of categories $F \colon \Iso{\B} \to \Iso{\A}$ with $F(\B) = \A$ and whose restriction to $\Aut(\B)$ is $H$.
\end{theorem}

Note that the inverse of $H$ is also continuous \cite[Exercise 2.3.5]{GaoBook}.

\begin{proof}
Define $F$ as before: 
Let $\Gamma$ be a map that assigns, to each copy $\Bhat$ of $\B$, an isomorphism $\Gamma^\Bhat\colon \Bhat\to\B$ with $\Gamma^{\B} = \id_{\B}$, and
(overloading notation a bit) assigns to each copy $\Ahat$ of $\A$, an isomorphism $\Gamma^\Ahat\colon \Ahat\to\A$ with $\Gamma^{\A} = \id_{\A}$.
For every copy $\Bhat$ of $\B$, we let $F(\Bhat)= \A$, and if $h \colon \Bhat\to\Bti$ is an isomorphism, then $F(h) = H({\Gamma^{\Bti}} \circ h \circ {\Gamma^\Bhat}^{-1})$.
Define $G$ in a similar way: For every copy $\Ahat$ of $\B$, we let $G(\Ahat)= \B$, and if $h \colon \Ahat\to\Ati$ is an isomorphism, then $G(h) = H^{-1}({\Gamma^{\Ati}} \circ h \circ {\Gamma^\Ahat}^{-1})$.

First, we want to show that $F$ and $G$ are inverse equivalences, via the natural isomorphisms $\eta$ and $\epsilon$ defined by
\[ \eta_{\Bhat}= \Gamma^{\Bhat} \colon \Bhat \to \B = G(F(\Bhat)) \text{ and } \epsilon_{\Ahat} = \Gamma^{\Ahat} \colon \Ahat \to \A = F(G(\Ahat)).\]
We have, by definition, $G(F(\Bhat)) = \B$ and $F(G(\Ahat)) = \A$. Now let $h \colon \Bhat \to \Bti$ be an isomorphism. Then
\begin{align*}
G \circ F(h) &= G(H({\Gamma^{\Bti}} \circ h \circ {\Gamma^\Bhat}^{-1})) \\
& = H^{-1}({\Gamma^{\A}} \circ H({\Gamma^{\Bti}} \circ h \circ {\Gamma^\Bhat}^{-1}) \circ {\Gamma^\A}^{-1}) \\
&= {\Gamma^{\Bti}} \circ h \circ {\Gamma^\Bhat}^{-1}.
\end{align*}
(Above, recall that  $\Gamma^{\A}=\id_{\A}$.)
So
\[ 
G(F(h)) \circ \eta_\Bhat = \eta_\Bti \circ h.
\]
Similarly, for an isomorphism $h \colon \Ahat \to \Ati$
\[ F(G(h)) \circ \epsilon_\Ahat = \epsilon_\Ati \circ h.\]
Thus $F \circ G$ and $G \circ F$ are naturally isomorphic to the identity.

Note that
\[ 
F(\eta_{\Bhat}) = H({\Gamma^{\B}} \circ \Gamma^{\Bhat} \circ {\Gamma^\Bhat}^{-1}) = H(\id_{\B}) = \id_{\A} = \Gamma^{\A} = \eta_{F(\Bhat)}.
\]
Similarly, $G(\epsilon_{\Ahat}) = \epsilon_{F(\Ahat)}$. Thus $F$, $G$, $\eta$, and $\epsilon$ form an adjoint equivalence of categories.
\end{proof}

\section{The construction}  \label{sec: construction}

In this section, we prove Theorem \ref{thm:main}. Let $\A$ and $\B$ be countable structures, and $F \colon \Iso{\B} \to \Iso{\A}$ a Baire-measurable functor. By Corollary \ref{cor:baire-to-borel}, up to natural isomorphism we may assume that $F$ is Borel.

The proof will involve a forcing: we will build multiple mutually generic structures and consider how the functor acts on the maps between these structures. The definability, in $\B$, of our forcing notion will give the formulas of our interpretation.

\subsection{The forcing notion}

Let $\B^{*}$ be the set of finite one-to-one tuples from $\B$. Since the domain of $\B$ is $\omega$, this is the same as finite tuples from $\omega$. We view $\B^{*}$ as a forcing notion, extension of tuples being extension of conditions. Thus, generics for these forcing notions are one-to-one functions  $\om\to\B$ respectively. A small amount of genericity guarantees these functions are onto and hence bijections.

Often in computable structure theory, forcing is used to build a single generic copy $\B_g$ of $\B$. Given a generic function $g \colon \omega \to \B$, $\B_g$ is the pullback of $\B$ along $g$. Here, we will want to build several generic copies and thus we will work with product forcing. Thus, given $\ell \in \omega$, we will define the product forcing $(\B^*)^\ell$. We write $p$ for a forcing condition in $(\B^*)^\ell$; $p$ is of the form $(\bbar_1,\ldots,\bbar_\ell)$.

We will want the forcing relation to be definable in $\B$. Often in computability theory, this is accomplished by taking as the forcing language $\L_{\omega_1 \omega}$ formulas about $\B$. Here, we will want to force statements of the form $F(\B_{g_1},g_1^{-1} \circ g_2,\B_{g_2})(i) = j$. Thus we will be required to force statements of the form $g_1^{-1} \circ g_2(i) = j$. This leads us to the definition of our forcing language.

\begin{definition}[Forcing language]
The finitary formulas in the forcing language for $(\B^*)^\ell$ are built up as follows:
\begin{itemize}
\item $\dot{g}_i^{-1} \circ \dot{g}_j (m) = n$ and $\dot{g}_i^{-1} \circ \dot{g}_j (m) \neq n$ where $m,n \in \omega$,
\item $R^{\B_{\dot{g}_i}}(a_1,\ldots,a_n)$ and $\neg R^{\B_{\dot{g}_i}}(a_1,\ldots,a_n)$ where $a_1,\ldots,a_n \in \omega$ and $R$ is a relation symbol in the language for $\B$,
\item finite conjunctions and finite disjunctions,
\item $\dot{g}_i(m) = n$ and $\dot{g}_i(m) \neq n$ where $m,n \in \omega$.
\end{itemize}
The forcing language $\L$ is built up from the finitary formulas by taking countable conjunctions and disjunctions. A formula is $X$-computable if the conjunctions and disjunctions are over $X$-c.e.\ sets of indices. By $\negat(\varphi)$, we mean the formal negation within the forcing language (flipping conjunctions and disjunctions, and negating the basic formulas).

We will also consider the restricted language $\L' \subset \L$ where we do not allow terms of the form $\dot{g}_i(m) = n$ or $\dot{g}_i(m) \neq n$.
\end{definition}

We use $\dot{g}_i$ as a formal symbol; the idea is that we will substitute a generic $g_i$ for $\dot{g}_i$. We will only get the definability of forcing within $\mathcal{B}$ for the restricted language $L'$; the whether or not the other sentences are forced depends on the presentation of $\mathcal{B}$.

We want to be able to express certain statements about $F$ by formulas in our forcing language. If we consider $F$ as a Borel functional, $F(\B_g)$ reads from its oracle statements about relations holding or not holding in $\B_g$ --- these are all in the forcing language --- and then computes its values using infinitary conjunctions and disjunctions. Thus, for $P$ a relation in the language of $\A$, we can express
\[
{F(\B_{g})}\models P(j_1,\ldots,j_{p(i)}) 
 \]
using infinitary formulas in the forcing language. Similarly, we can express
\[ F(\B_{g_1}, g_2^{-1} \circ g_1, \B_{g_2})(i)  = j.\]
If $F$ is $\Delta^0_\alpha$, then we can express these as $\Dec_\alpha$ formulas (i.e., as $\Sic_\alpha$ formulas and also as $\Pic_\alpha$ formulas). Similarly, if $F$ is $\bfDelta^0_\alpha$, then we can express these as $\Dei_\alpha$ formulas.
Using conjunctions and disjunctions of such statements, we can also express more complicated statements such as
\[ F(\B_{g_2}, g_1^{-1}\circ g_2, \B_{g_1}) = F(\B_{g_1}, g_2^{-1}\circ g_1, \B_{g_2})^{-1} \]
and
\[ F(\B_{g_2}, g_3^{-1}\circ g_2, \B_{g_3}) \circ F(\B_{g_1}, g_2^{-1}\circ g_1, \B_{g_2}) = F(\B_{g_1}, g_3^{-1}\circ g_1, \B_{g_3}) \]
in the forcing language. These formulas are all in the restricted language $\L'$. In the language $\L$, we can express 
\[ F(\Bhat, g_1^{-1}, \Bhat_{g_1})(i) = j.\]
If $F$ is $\Delta^0_\alpha$, then this is a $\Bhat$-computable $\Dec_\alpha$ formula.

\begin{definition}[Definition of Forcing]\label{def:forcing}
Let $p = (\bbar_1,\ldots,\bbar_\ell)$ be a forcing condition for $(\B^*)^\ell$. We define $p \forces_{(\B^*)^\ell} \varphi$ for $\varphi$ a sentence of the forcing language. We begin with the finitary formulas.
\begin{itemize}
\item if $\varphi \equiv \dot{g}_i^{-1} \circ \dot{g}_j (m) = n$, then $p \forces_{(\B^*)^\ell} \varphi$ if and only if $\bbar_i(n)$ and $\bbar_j(m)$ are defined and equal.
\item if $\varphi \equiv \dot{g}_i^{-1} \circ \dot{g}_j (m) \neq n$, then $p \forces_{(\B^*)^\ell} \varphi$ if and only if either:
	\begin{itemize}
		\item $\bbar_i(n)$ and $\bbar_j(m)$ are defined and distinct, or
		\item there is $m' \neq m$ such that $\bbar_i(n) = \bbar_j(m')$, or
		\item there is $n' \neq n$ such that $\bbar_i(n') = \bbar_j(m)$.
	\end{itemize}
\item if $\varphi \equiv R^{\B_{\dot{g}_i}}(a_1,\ldots,a_n)$, then $p \forces_{(\B^*)^\ell} \varphi$ if and only if $\bbar_i(a_1),\ldots,\bbar_i(a_n)$ are all defined and $\B \models R(\bbar_i(a_1),\ldots,\bbar_j(a_n))$.
\item if $\varphi \equiv \neg R^{\B_{\dot{g}_i}}(a_1,\ldots,a_n)$, then $p \forces_{(\B^*)^\ell} \varphi$ if and only if $\bbar_i(a_1),\ldots,\bbar_i(a_n)$ are all defined and $\B \models \neg R(\bbar_i(a_1),\ldots,\bbar_j(a_n))$.
\item if $\varphi \equiv \dot{g}_i(m) = n$, then $p \forces_{(\B^*)^\ell} \varphi$ if and only if $\bbar_i(m) = n$.
\item if $\varphi \equiv \dot{g}_i(m) \neq n$, then $p \forces_{(\B^*)^\ell} \varphi$ if and only if either $\bbar_i(m) \neq n$, or for some $m' \neq m$, $\bbar_i(m') = n$.
\item if $\varphi \equiv \psi_1 \vee \cdots \vee \psi_n$, then $p \forces_{(\B^*)^\ell} \varphi$ if and only if $p \forces \psi_i$ for some $i$.
\item if $\varphi \equiv \psi_1 \wedge \cdots \wedge \psi_n$, then $p \forces_{(\B^*)^\ell} \varphi$ if and only if $p \forces \psi_i$ for each $i$.
\end{itemize}
Now for infinitary formulas:
\begin{itemize}
\item if $\varphi\equiv \bigvee_n\psi_n$, then $p \forces_{(\B^*)^\ell} \bigvee_n \psi_n$ if and only if there is $n$ such that $p \forces_{(\B^*)^\ell} \psi_n$.
\item if $\varphi\equiv \bigwedge_n \psi_n$, then $p \forces_{(\B^*)^\ell} \bigwedge_n \psi_n$ if for all $n$ and $q \supseteq p$, there is $r \supseteq q$ such that $r \forces_{(\B^*)^\ell} \psi_n$.
\end{itemize}
\end{definition}

Given an injection $g \colon \omega \to \omega$, we can define a structure $\B_g$ using the pullback of $\B$ along $g$. That is, $R^{\B_g}(a_1,\ldots,a_n)$ if and only if $R^{\B}(g(a_1),\ldots,g(a_n))$. If $g$ is a bijection, then $\B_g$ is isomorphic to $\B$ via $g \colon \B_g \to \B$.

Given $\varphi$ a sentence in the forcing language for $(\B^*)^\ell$, and $g_1,\ldots,g_\ell$ functions $\omega \to \B$, we say that $\varphi[g_1,\ldots,g_\ell]$ \textit{holds} if $\varphi$ becomes true under the natural interpretation, substituting $g_i$ for $\dot{g}_i$.

\begin{lemma}\label{lem:ext}
If $p \forces_{(\B^*)^\ell} \varphi$, and $q \supseteq p$, then $q \forces_{(\B^*)^\ell} \varphi$.
\end{lemma}
\begin{proof}
The proof is by induction on the complexity of $\varphi$. The lemma is clear for the finitary formulas. If $\varphi \equiv \bigvee_n\psi_n$ and $p \forces_{(\B^*)^\ell} \varphi$, then there is $n$ such that $p \forces_{(\B^*)^\ell} \psi_n$. By the induction hypothesis, $q \forces_{(\B^*)^\ell} \psi_n$. If $\varphi \equiv \bigwedge_n\psi_n$, then for all $n$ and $r \supseteq q$, $r \supseteq p$, and so there is $r' \supseteq r$ such that $r' \forces_{(\B^*)^\ell} \psi_n$. Thus $q \forces_{(\B^*)^\ell} \varphi$.
\end{proof}

\begin{lemma}\label{lem:decides}
For every $p$ and $\varphi$, there is $q \supseteq p$ such that $q$ decides $\varphi$.
\end{lemma}
\begin{proof}
The proof is by induction. It is easy to see that the lemma holds when $\varphi$ is a finitary formula. If $\varphi\equiv \bigvee_n\psi_n$, then if there are $n$ and $q \supseteq p$ such that $q \forces_{(\B^*)^\ell} \psi_n$, then we are done. Otherwise, for all $n$ and $q \supseteq p$, $q \nforces_{(\B^*)^\ell} \psi_n$. By the induction hypothesis, there is $r \supseteq q$ such that $r$ decides $\psi_n$; by the previous lemma, $r \forces_{(\B^*)^\ell} \negat(\psi_n)$. Thus $p \forces_{(\B^*)^\ell} \negat(\varphi) \equiv \bigwedge_n \negat(\psi_n)$. The same argument works if $\varphi\equiv \bigwedge_n\psi_n$.
\end{proof}

\begin{lemma}\label{lem:neg}
It is not the case that $p \forces_{(\B^*)^\ell} \varphi$ and $p \forces_{(\B^*)^\ell} \negat(\varphi)$.
\end{lemma}
\begin{proof}
The lemma is easy to check for finitary formulas. If $\varphi\equiv \bigvee_n\psi_n$ or $\varphi \equiv \bigwedge_n \psi_n$, and $p \forces_{(\B^*)^\ell} \varphi$ and $p \forces_{(\B^*)^\ell} \negat(\varphi)$, then there is $n$ such that $p \forces_{(\B^*)^\ell} \psi_n$. Also, there is $q \supseteq p$ such that $q \forces_{(\B^*)^\ell} \negat(\psi_n)$. This contradicts the induction hypothesis.
\end{proof}

\begin{definition}
Let $X \subseteq \omega$. By an $X$-generic for $(\B^*)^\ell$ we mean a tuple $\boldsymbol{g} = (g_1,\ldots,g_\ell)$ of mutually $(X \oplus \B)$-hyperarithmetically generic functions $\omega \to \B$.
\end{definition}

It is clear that for any particular $X$-computable sentence of the forcing language, the forcing relation is $X \oplus \B$-hyperarithmetic. Thus, by Lemma \ref{lem:decides}, an $X$-generic $\boldsymbol{g}$ for $(\B^*)^\ell$ has the property that it forces every $X$-computable sentence or its negation. We also get that $(\boldsymbol{g}_1,\boldsymbol{g}_2)$ is $X$-generic for $(\B^*)^{\ell_1 + \ell_2}$ if and only if $\boldsymbol{g}_1$ is $X$-generic for $(\B^*)^{\ell_1}$ and $\boldsymbol{g}_2$ is $X \oplus \boldsymbol{g}_1$-generic for $(\B^*)^{\ell_2}$.

\begin{lemma}[Restriction]\label{lem:restriction}
If $\varphi$ is a computable sentence of the forcing language which does not involve $g_i$, then $(\bbar_1,\ldots,\bbar_\ell) \forces_{(\B^*)^\ell} \varphi$ if and only if $(\bbar_1,\ldots,\bbar_{i-1},\bbar_{i+1},\ldots,\bbar_\ell) \forces_{(\B^*)^{\ell-1}} \varphi$.
\end{lemma}
\begin{proof}
This is a simple induction argument.
\end{proof}

\begin{lemma}[Forcing Lemma]
Let $\varphi$ be an $X$-computable sentence of the forcing language for $(\B^*)^\ell$.
\begin{enumerate}
\item For $X$-generic $\boldsymbol{g}$, $\varphi[\boldsymbol{g}]$ holds if and only if for some $p \subset \boldsymbol{g}$, $p \forces_{(\B^*)^\ell} \varphi$.
\item If $\varphi$ starts with a $\bigwedge$, then $p \forces_{(\B^*)^\ell} \varphi$ if and only if for every $X$-generic $\boldsymbol{g} \supset p$, $\varphi[\boldsymbol{g}]$ holds.
\end{enumerate}
\end{lemma}
\begin{proof}
For (1), first suppose that for some $p = (\bbar_1,\ldots,\bbar_\ell) \subseteq \boldsymbol{g} = (g_1,\ldots,g_\ell)$, $p \forces \varphi$. We argue by induction. For the finitary formulas, everything is simple:
\begin{itemize}
\item if $\varphi \equiv \dot{g}_i^{-1} \circ \dot{g}_j (m) = n$, then $\bbar_i(n) = \bbar_j(m)$ and so $g_i^{-1} \circ g_j (m) = n$.
\item if $\varphi \equiv \dot{g}_i^{-1} \circ \dot{g}_j (m) \neq n$, then either:
	\begin{itemize}
		\item $\bbar_i(n) \neq \bbar_j(m)$ and so $g_i^{-1} \circ g_j (m) \neq n$,
		\item there is $m' \neq m$ such that $\bbar_i(n) = \bbar_j(m')$, and so since $g_j$ is injective, $g_i^{-1} \circ g_j (m) \neq n$, or
		\item there is $n' \neq n$ such that $\bbar_i(n') = \bbar_j(m)$, and so since $g_i$ is injective, $g_i^{-1} \circ g_j (m) \neq n$.
	\end{itemize}
\item if $\varphi \equiv R^{\B_{\dot{g}_i}}(a_1,\ldots,a_n)$, then $\B \models R(\bbar_i(a_1),\ldots,\bbar_i(a_n))$ and so $\B_{g_i} \models R(a_1,\ldots,a_n)$.
\item if $\varphi \equiv \neg R^{\B_{\dot{g}_i}}(a_1,\ldots,a_n)$, then $\B \models \neg R(\bbar_i(a_1),\ldots,\bbar_i(a_n))$ and so $\B_{g_i} \models \neg R(a_1,\ldots,a_n)$.
\item if $\varphi \equiv \dot{g}_i(m) = n$, then $\bbar_i(m) = n$ and so $g_i(m) = n$.
\item if $\varphi \equiv \dot{g}_i(m) \neq n$, then either $\bbar_i(m) \neq n$, or for some $m \neq m'$, $\bbar_i(m') = n$; thus $g_i(m) \neq n$.
\item if $\varphi \equiv \psi_1 \vee \cdots \vee \psi_n$, then $p \forces_{(\B^*)^\ell} \psi_i$ for some $i$ and so $\psi_i[\boldsymbol{g}]$ holds for some $i$.
\item if $\varphi \equiv \psi_1 \wedge \cdots \wedge \psi_n$, then $p \forces_{(\B^*)^\ell} \psi_i$ for all $i$ and hence for all $i$, $\psi_i[\boldsymbol{g}]$ holds.
\end{itemize}
Now for infinitary formulas:
\begin{itemize}
\item if $\varphi\equiv \bigvee_n\psi_n$, then $p \forces_{(\B^*)^\ell} \psi_i$ for some $i$ and so $\psi_i[\boldsymbol{g}]$ holds for some $i$.
\item if $\varphi\equiv \bigwedge_n \psi_n$, then for all $n$ and $q \supseteq p$, there exists $r \supseteq q$ such that $r \forces_{(\B^*)^\ell} \psi_n$. Fix $n$. Since $\boldsymbol{g}$ is generic, there is $q \subset \boldsymbol{g}$ such that $q$ decides $\psi_n$. We may assume, by Lemma \ref{lem:ext}, that $q \supseteq p$. So there is $r \supseteq q$ such that $r \forces_{(\B^*)^\ell} \psi_n$. By Lemma \ref{lem:neg}, $q \forces_{(\B^*)^\ell} \psi_n$. By the induction hypothesis, $\psi_n[\boldsymbol{g}]$ holds. Since this was true for all $n$, $\varphi[\boldsymbol{g}]$ holds.
\end{itemize}

Now suppose that $\varphi[\boldsymbol{g}]$ holds. There is $p \subseteq \boldsymbol{g}$ such that $p$ decides $\varphi$. If $p \forces_{(\B^*)^\ell} \negat(\varphi)$, then $\negat(\varphi)[\boldsymbol{g}]$ holds. This is a contradiction. Hence $p \forces_{(\B^*)^\ell} \varphi$.

For (2), suppose that $p \forces_{(\B^*)^\ell} \varphi$. Let $\boldsymbol{g} \supseteq p$ be $X$-generic. By (1), $\varphi[\boldsymbol{g}]$ holds.

For the other direction, suppose that for all $X$-generic $\boldsymbol{g} \supseteq p$, $\varphi[\boldsymbol{g}]$ holds. Then for all $q \supseteq p$, $q \nforces_{(\B^*)^\ell} \negat(\varphi)$; if we did have $q \forces_{(\B^*)^\ell} \negat(\varphi)$, then for some $X$-generic $\boldsymbol{g} \supseteq q$, $\negat(\varphi)[\boldsymbol{g}]$ would hold, a contradiction. Now if $\varphi$ begins with $\bigwedge$, say $\varphi \equiv \bigwedge_n \psi_n$, then $\negat(\varphi) \equiv \bigvee_n \negat(\psi_n)$. So for all $n$ and $q \supseteq p$, $q \nforces_{(\B^*)^\ell} \negat(\psi_n)$. Now by Lemma \ref{lem:decides} there is $r \supseteq q$ such that $r$ decides $\psi_n$; we cannot have $r \forces_{(\B^*)^\ell} \negat(\psi_n)$ (since $r \supseteq p$) and so $r \forces_{(\B^*)^\ell} \psi_n$. Thus $p \forces_{(\B^*)^\ell} \varphi$.
\end{proof}

\begin{lemma}[Definability of Forcing]\label{lem:definability}
For $\alpha\geq 1$, given a $\Si_\alpha$ formula $\varphi$ in the restricted language $\L'$, the set $\{p\in (\B^*)^{\ell}: p\forces \varphi\}$ is $\Sic_\a$-definable in $\B$, and if $\varphi$ is $\Pi_\alpha$, $\{p\in (\B^*)^{\ell}: p\forces \varphi\}$ is $\Pic_\a$-definable. This also relativizes.
\end{lemma}
\begin{proof}
We argue by induction. For finitary formulas $\varphi$, it is easy to see from Definition \ref{def:forcing} that the set $\{p\in (\B^*)^{\ell}: p\forces_{(\B^*)^\ell} \varphi\}$ is definable in $\B$ by a finitary formula. The key is to note that the tuples $\bbar$ and $\cbar$ such that $\bbar(n) = \cbar(m)$, or the tuples $\bbar$ such that $\B \models R(\bbar(a_1),\ldots,\bbar(a_n))$, are definable in $\B$ by atomic formulas.

Now we consider infinitary formulas. If $\varphi\equiv \bigvee_n\psi_n$, then $p \forces_{(\B^*)^\ell} \bigvee_n \psi_n$ if and only if for some $n$, $p \forces_{(\B^*)^\ell} \psi_n$. Since, for each $n$, $p \forces_{(\B^*)^\ell} \psi_n$ is $\Pic_\beta$-definable in $\B$ for some $\beta < \alpha$, this is $\Sic_\a$-definable in $\B$.

If $\varphi\equiv \bigwedge_n \psi_n$, then note that by Lemmas \ref{lem:ext}, \ref{lem:decides}, and \ref{lem:neg}, $p \forces \bigwedge_n \psi_n$ if and only if for all $q \supseteq p$, $q \nforces_{(\B^*)^\ell} \negat(\psi_n)$; $q \nforces_{(\B^*)^\ell} \negat(\psi_n)$ is $\Sic_{\beta}$-definable in $\B$ for some $\beta < \alpha$, and so $p \forces_{(\B^*)^\ell} \bigwedge_n \psi_n$ is $\Pic_\a$-definable in $\B$.
\end{proof}

\subsection{The definition of the interpretation}

Recall that $F\colon \Iso{\B} \to \Iso{\A}$ is a Borel functor. As everything will relativize, we will assume from now on that it is a lightface Borel operator.

\begin{definition}
We define the domain of interpretation, $\Dom{\B}{\A}$, as a subset of $\B^{*} \times \omega$ as follows:
For $(\bar{b},i) \in \B^{*} \times \omega$, let
\[
(\bar{b},i) \in \Dom{\B}{\A}   \iff  (\bbar,\bbar) \forces_{(\B^{*})^2}   F(\B_{\dot{g}_1}, \dot{g}_2\circ \dot{g}_1, \B_{\dot{g}_2})(i)  = i. 
\]
Recall that subsets of $\B^{<\omega}\times \omega$ can be effectively coded by subsets of $\B^{<\om}$.
Next, we define a relation $\sim$ on $\Dom{\B}{\A}$ which we will later prove is an equivalence relations.
For  $(\bar{b},i), (\bar{c},j) \in \Dom{\B}{\A}$, let 
\[
(\bar{b},i) \sim (\bar{c},j)
	\iff
(\bbar,\cbar) \forces_{(\B^{*})^2}   F(\B_{\dot{g}_1}, \dot{g}_2\circ \dot{g}_1, \B_{\dot{g}_2})(i)  = j.
\] 
Last, we need to interpret the relation symbols.
For each relation symbol $P_i$ of arity $p(i)$ in the language of $\A$, we define a relation $R_i$ on $\Dom{\B}{\A}$ as follows:
For $(\bar{b}_1,k_1),\ldots,(\bar{b}_{p(i)},k_{p(i)}) \in  \Dom{\B}{\A}$, let 
\begin{multline*}
((\bar{b}_1,k_1),\ldots,(\bar{b}_{p(i)},k_{p(i)}))\in R_i
		\iff  
(\exists\bar{c}\in\B^{*})(\exists j_1,\ldots,j_{p(i)} \in \omega)\\ 
\left( \bigwedge_{s=1}^{p(i)}(\bar{b}_s,k_s) \sim (\bar{c},j_s)  \right) \and\left( \cbar \forces_{\B^{*}} (j_1,\ldots,j_{p(i)}) \in P_i^{F(\B_{\dot{g}})}\right).
\end{multline*}
\end{definition}

\noindent By Lemma \ref{lem:definability}, these are all defined by formulas of $\L_{\omega_1 \omega}$ since they can be expressed in $L'$.

\subsection{Verifications}

The first thing to observe before starting the verifications is that since $F$ is a functor that works for all copies of $\B$, all its properties are forced by the empty conditions.
For instance,
$$(\emptyset,\emptyset,\emptyset)\forces_{\B^{*3}} F(\B_{\dot{g}_2}, \dot{g}_3^{-1}\circ \dot{g}_2, \B_{\dot{g}_3}) \circ F(\B_{\dot{g}_1}, \dot{g}_2^{-1}\circ \dot{g}_1, \B_{\dot{g}_2}) = F(\B_{\dot{g}_1}, \dot{g}_3^{-1}\circ \dot{g}_1, \B_{\dot{g}_3}).$$

\begin{lemma}
$\sim$ is an equivalence relation on $\Dom{\B}{\A}$,
\end{lemma}
\begin{proof}
Reflexivity follows from the definition of $\Dom{\B}{\A}$.
Symmetry holds because  $(\emptyset,\emptyset)\forces_{\B^{*2}} F(\B_{\dot{g}_2}, \dot{g}_1^{-1}\circ \dot{g}_2, \B_{\dot{g}_1}) = F(\B_{\dot{g}_1}, \dot{g}_2^{-1}\circ \dot{g}_1, \B_{\dot{g}_2})^{-1}$.
Transitivity follows from the fact that $(\emptyset,\emptyset,\emptyset)\forces_{\B^{*3}} F(\B_{\dot{g}_2}, \dot{g}_3^{-1}\circ \dot{g}_2, \B_{\dot{g}_3}) \circ F(\B_{\dot{g}_1}, \dot{g}_2^{-1}\circ \dot{g}_1, \B_{\dot{g}_2}) = F(\B_{\dot{g}_1}, \dot{g}_3^{-1}\circ \dot{g}_1, \B_{\dot{g}_3})$.
\end{proof}

The next objective of this subsection is to define a map $\Fra \colon A\to \Dom{\B}{\A}$ which gives an isomorphism between $\A$ and its interpretation within $\B$.
Remember we are fixing a copy of $\B$, and that $\A=F(\B)$.

Let $g\colon\om\to\B$ be generic; before defining $\Fra$, we define a map $\Fra_g\colon F(\B_g) \to \Dom{\B}{\A}$ also intended to be an isomorphism (Lemma \ref{lem:iso}).
Given $i$, we let $\Fra_g(i)$ be the least tuple of the form $(\cbar, i)$ for $\cbar\subset g$, and with $(\cbar,i)\in \Dom{\B}{\A}$ (we prove such a tuple exists in Lemma \ref{lem:exists}).
We will need to show that all of this works (Lemmas \ref{lem:onto} and \ref{lem:unique}).
Then, to define $\Fra$, we simply compose $\Fra_g\colon F(\B_g)\to\Dom{\B}{\A}$ with $F(\B, g^{-1}, \B_g)\colon \A\to F(\B_g)$.
We will also need to show that this definition is independent of the choice of $g$ (Lemma \ref{lem:indep-g}).

The first lemma shows that $\Fra_g(i)$ is defined for every $i$.

\begin{lemma}\label{lem:exists}
For every generic $g\colon\om\to\B$ and every $i\in\om$ there exists $n\in\om$ such that $(g\upto n, i)\in \Dom{\B}{\A}$.
\end{lemma}
\begin{proof}
Let $g_2$ be generic with respect to $g = g_1$ so that $(g_1,g_2)$ is generic for $(\B^*)^2$. Let $j= F(\B_{g_1}, g_2^{-1}\circ g_1, \B_{g_2})(i)$.
For some $\bbar\subset g_1$ and some $\cbar\subset g_2$, $(\bbar,\cbar)\forces_{(\B^*)^2} F(\B_{\dot{g}_1}, \dot{g}_2^{-1}\circ \dot{g}_1, \B_{\dot{g}_2})(i) = j$.
Notice that we also have $(\bbar,\cbar)\forces_{(\B^*)^2}  F( \B_{\dot{g}_2}, \dot{g}_1^{-1}\circ \dot{g}_2,\B_{\dot{g}_1})(j ) = i$.
It then follows that
\[
(\bbar,\cbar,\bbar)\forces_{(\B^*)^3} F(\B_{\dot{g}_1}, \dot{g}_2^{-1}\circ \dot{g}_1, \B_{\dot{g}_2})(i) = j  \and F( \B_{\dot{g}_2}, \dot{g}_3^{-1}\circ \dot{g}_2,\B_{\dot{g}_3})(j ) = i,
\]
and hence 
\[
(\bbar,\cbar,\bbar)\forces_{(\B^*)^3} F(\B_{\dot{g}_1}, \dot{g}_3^{-1}\circ \dot{g}_1, \B_{\dot{g}_3})(i) = i.
\]
Since $g_2$ does not appear in the formula above, by Lemma \ref{lem:restriction} we get
$$(\bbar,\bbar)\forces_{(\B^*)^2} F(\B_{\dot{g}_1}, \dot{g}_2^{-1}\circ \dot{g}_1, \B_{\dot{g}_2})(i) = i,$$
and hence that $(\bbar,i)\in \Dom{\B}{\A}$.
\end{proof}

The second lemma shows that $\Fra_g$ is onto the set of $\sim$-equivalence classes.

\begin{lemma}\label{lem:onto}
For every generic $g\colon\om\to\B$ and every $(\cbar,j)\in\Dom{\B}{\A}$ there exists $n\in\om$ and $i\in\om$ such that $(g\upto n, i)\sim (\cbar,j)$.
\end{lemma}
\begin{proof}
The proof is similar to that of the lemma above. 
Let $g_2 \supseteq \bar{c}$ be generic with respect to $g=g_1$, and let $j= F(\B_{g_1}, g_2^{-1}\circ g_1, \B_{g_2})(i)$.
There are $\bbar\subset g_1$ and $\cbar'$ with $\cbar \subseteq \cbar' \subset g_2$ such that $(\bbar,\cbar')\forces_{(\B^*)^2} F(\B_{\dot{g}_1}, \dot{g}_2^{-1}\circ \dot{g}_1, \B_{\dot{g}_2})(i) = j$ and also $(\bbar,i)\in \Dom{\B}{\A}$.
Since $(\cbar,j) \in \Dom{\B}{\A}$, we see that $(\cbar,\cbar) \forces_{(\B^*)^2} F(\B_{\dot{g}_1}, \dot{g}_2^{-1}\circ \dot{g}_1, \B_{\dot{g}_2})(j) = j$. Then
\[ 
(\bbar,\cbar',\cbar) \forces_{(\B^*)^3} F(\B_{\dot{g}_1}, \dot{g}_2^{-1}\circ \dot{g}_1, \B_{\dot{g}_2})(i) = j  \and F( \B_{\dot{g}_2}, \dot{g}_3^{-1}\circ \dot{g}_2,\B_{\dot{g}_3})(j) = j
\]
and hence
\[ (\bbar,\cbar',\cbar) \forces_{(\B^*)^3} F(\B_{\dot{g}_1}, \dot{g}_3^{-1}\circ \dot{g}_1, \B_{\dot{g}_3})(i) = j.\]
Since $g_2$ does not appear in the formula above, $(\bbar,\cbar) \forces_{(\B^*)^2} F(\B_{\dot{g}_1}, \dot{g}_3^{-1}\circ \dot{g}_1, \B_{\dot{g}_3})(i) = j$ and so $(\bbar, i)\sim (\cbar,j)$.
\end{proof}

The third lemma shows that $\Fra_g$ is one-to-one on $\sim$-equivalence classes.

\begin{lemma}\label{lem:unique}
For $(\cbar,i)$, $(\dbar,j)\in \Dom{\B}{\A}$ with $\cbar\subseteq \dbar$ we have that  $(\cbar,i)  \sim  (\dbar,j)$ if and only if $i=j$.
\end{lemma}
\begin{proof}
By definition, $(\cbar,i)  \sim  (\dbar,j)$ if and only if $(\cbar,\dbar)\forces_{(\B^*)^2} F(\B_{\dot{g}_1}, \dot{g}_2^{-1}\circ \dot{g}_1, \B_{\dot{g}_2})(i) = j$.
But since $(\cbar,i)\in \Dom{\B}{\A}$, we know $(\cbar,\cbar)\forces_{(\B^*)^2} F(\B_{\dot{g}_1}, \dot{g}_2^{-1}\circ \dot{g}_1, \B_{\dot{g}_2})(i) = i$.
With $(\cbar,\dbar)$ extending $(\cbar,\cbar)$, we get that $(\cbar,i)  \sim  (\dbar,j)$ if and only if $i=j$.
\end{proof}

So we have that $\Fra_g$ is a bijection from $\om$ onto $\Dom{\B}{\A}/\sim$.
We now show that it is an isomorphism from $F(\B_g)$ to $(\Dom{\B}{\A} /\!\sim; ~R_0/\!\sim,~R_1/\!\sim,...)$.

\begin{lemma}\label{lem:iso}
For every relation symbol $P_i$, and $(j_1,...,j_{p(i)})\in \om^{p(i)}$, $F(\B_g)\models P_i(j_1,...,j_{p(i)}) \iff (\Fra_g(j_1),...,\Fra_g(j_{p(i)})) \in R_i$. 
\end{lemma}
\begin{proof}
First suppose that $F(\B_g)\models P_i(j_1,...,j_{p(i)})$. Then there is $\bar{c} \subseteq g$ such that $\bar{c} \forces_{\B^*} F(\B_{\dot{g}}) \models P_i(j_1,...,j_{p(i)})$; by Lemma \ref{lem:exists} we may also assume that $(\bar{c},j_s) \in \Dom{\B}{\A}$ for each $s$. Then by Lemma \ref{lem:unique}, $\Fra_g(j_s) \sim (\bar{c},j_s)$. Hence, by definition of $R_i$, $(\Fra_g(j_1),...,\Fra_g(j_{p(i)})) \in R_i$.

On the other hand, suppose that $(\Fra_g(j_1),...,\Fra_g(j_{p(i)})) \in R_i$. Then there are $\bar{c} \in \B^*$ and $k_1,\ldots,k_{p(i)}$ such that for each $s$, $\Fra_g(j_s) \sim (\bar{c},k_s)$ and $\bar{c} \forces_{\B^*} F(\B_{\dot{g}_1})\models P_i(k_1,...,k_{p(i)})$. Since $\Fra_g(j_s) \sim (\bar{c},k_s)$, there is $\bar{d} \subseteq g$ such that for each $s$, $(\bar{c},\bar{d}) \forces_{(\B^*)^2} F(\B_{\dot{g}_1}, \dot{g}_2^{-1}\circ \dot{g}_1, \B_{\dot{g}_2})(k_s) = j_s$. Then $(\bar{c},\bar{d}) \forces_{(\B^*)^2} F(\B_{\dot{g}_2}) \models P_i(j_1,...,j_{p(i)})$. Since $\dot{g}_1$ does not appear in this formula, $\bar{d} \forces F(\B_{\dot{g}}) \models P_i(j_1,...,j_{p(i)})$. But $g \supseteq \bar{d}$ is generic, so $F(\B_g) \models P_i(j_1,...,j_{p(i)})$.
\end{proof}

Last, we need to show that $\Fra$, defined as $\Fra_g\circ F(\B, g^{-1}, \B_g)$, is independent of the choice of the generic $g$.

\begin{lemma}\label{lem:indep-g}
For $i\in\om$ and $(\cbar,j)\in \Dom{\B}{\A}$, 
\[
\Fra(i) \sim (\cbar,j)    \iff    \cbar\forces_{\B^*} F(\B, \dot{g}_1^{-1}, \B_{\dot{g}_1})(i) = j.
\]
\end{lemma}
\begin{proof}
Let $g_2 \supset \cbar$ be generic relative to $g = g_1$. Let $k = F(\B, g_1^{-1}, \B_{g_1})(i)$ and $(\dbar,k)=\Fra(i) = \Fra_{g_1}(k)$. For some $\dbar' \supseteq \dbar$, $\dbar' \forces_{\B^*} F(\B, \dot{g}_1^{-1}, \B_{\dot{g}_1})(i) = k$.

Suppose that $\Fra(i) \sim (\cbar,j)$. Then $(\dbar,k) \sim (\cbar,j)$ and so $(\dbar, \cbar)\forces_{(\B^*)^2} F(\B_{\dot{g}_1}, {\dot{g}_2}^{-1}\circ \dot{g}_1, \B_{\dot{g}_2})(k)=j$.
We see that $(\dbar',\cbar) \forces_{(\B^*)^2} F(\B, {\dot{g}_2}^{-1}, \B_{\dot{g}_2})(i)=j$, and since $\dot{g}_1$ does not appear in this formula, $\cbar\forces_{\B^*} F(\B, \dot{g}_1^{-1}, \B_{\dot{g}_1})(i) = j$ as desired.

Now suppose that $\cbar \forces_{\B^*} F(\B, \dot{g}_1^{-1}, \B_{\dot{g}_1})(i) = j$. Then since $\dbar' \forces_{\B^*} F(\B, \dot{g}_1^{-1}, \B_{\dot{g}_1})(i) = k$, $(\cbar,\dbar') \forces_{(\B^*)^2} F(\B_{\dot{g}_1}, {\dot{g}_2}^{-1}\circ \dot{g}_1, \B_{\dot{g}_2})(k)=j$. Thus $(\cbar,j) \sim (\dbar',k)$ and $(\dbar',k) \sim (\dbar,k)$ by Lemma \ref{lem:unique}. 
\end{proof}

In fact, given any $\Btilde \cong \B$, we can define $\Fra^{\Btilde} \colon F(\Btilde) \to \Dom{\Btilde}{\A}$ by 
\[
\Fra^{\Btilde}(i) \sim (\cbar,j)    \iff    \cbar\forces_{\B^*} F(\Btilde, \dot{g}_1^{-1}, \Bti_{\dot{g}_1})(i) = j.
\]
This gives an isomorphism from $F(\Btilde)$ to $(\Dom{\Btilde}{\A} / \!\sim;~R_0/\!\sim,~R_1/\!\sim,...)$. If $F$ is $\Delta^0_\alpha$, then $\Fra^{\Btilde}$ is $\Delta^0_\alpha(\Bti)$ uniformly in $\Bti$.

\begin{lemma}\label{lem:nat-iso}
There is a natural isomorphism between $F$ and $F_{\I}$.
\end{lemma}
\begin{proof}
Recall that given $\Btilde \cong \B$, we build $F_{\I}(\Btilde)$ out of the interpretation of $\A$ within $\Btilde$ by pulling back through a bijection $\tau\colon \om\to \Dom{\Btilde}{\A} / \!\sim$.
Let us call this bijection $\tau^\Btilde$; it gives a well-defined isomorphism from $F_{\I}(\Btilde)$ to $\Dom{\Btilde}{\A}/\!\sim$.
We define
\[
\eta_{\Btilde} = (\tau^{\Btilde})^{ -1} \circ \Fra^{\Btilde} \colon F(\Btilde) \to F_{\I}(\Btilde).
\]
We need to show that $\eta$ is a natural isomorphism.
It is clear that $\eta_{\Btilde}$ is an isomorphism.
We must prove that, for all $\Btilde,\Bhat \in \Iso{\B}$ and all isomorphisms $h\colon \Btilde \to \Bhat$, the following diagram commutes:
\[
\xymatrix{
F(\Btilde)\ar[d]_{F(h)}  \ar[r]^{\Fra^{\Btilde}}    \ar@/^2pc/[rr]^{\eta_{\Btilde}}       & \Dom{\Btilde}{\A}\ar[d]_{\tilde{h}}        &     F_{\I}(\Btilde)\ar[d]^{F_{\I}(h)}   \ar[l]_{\tau^{\Btilde}} \\
F(\Bhat)   \ar[r]_{\Fra^{\Bhat}}       \ar@/_2pc/[rr]_{\eta_{\Bhat}}          & \Dom{\Bhat}{\A}   & F_{\I}(\Bhat)      \ar[l]^{\tau^{\Bhat}}
}\]
Here, $\tilde{h}\colon \Dom{\Btilde}{\A} \to \Dom{\Bhat}{\A}$ is the restriction of $h\colon \Btilde^{<\om} \to \Bhat^{<\om}$, which is the extension of $h\colon \Btilde \to \Bhat$.

The right-hand square commutes by definition of $F_{\I}(h)$. 
To show that the left-hand square commutes, take $i \in F(\Btilde)$ and $j = F(h)(i) \in F(\Bhat)$.
Let $(\cbar,i')=\Fra^{\Btilde}(i) \in \Dom{\Btilde}{\A}$. We must show that $\Fra^{\Bhat}(j) = h(\cbar, i') = (h(\cbar),i')$.
Since $(\cbar,i')=\Fra^{\Btilde}(i)$,
\[ \cbar\forces_{\B^*} F(\Btilde, \dot{g}^{-1}, \Btilde_{\dot{g}})(i) = i'.\]
We claim that
\[ h(\cbar) \forces_{\B^*} F(\Bhat, \dot{g}^{-1}, \Bhat_{\dot{g}})(j) = i' \]
from which it follows that $\Fra^{\Bhat}(j) = (h(\cbar),i')$.
Let $g \supset h(\cbar)$ be $\Btilde \oplus \Bhat \oplus h$-generic. Then $h^{-1} \circ g \supset \bar{c}$ is also $\Btilde \oplus \Bhat \oplus h$-generic. So
\[ F(\Btilde, g^{-1} \circ h, \Btilde_{h^{-1} \circ g})(i) = i'. \]
Then
\[ F(\Btilde, g^{-1} \circ h, \Btilde_{h^{-1} \circ g}) \circ F(\Bhat, h^{-1}, \Btilde) \circ F(\Btilde, h, \Bhat) (i) = i'.\]
Simplifying this, and using the fact that $F(\Btilde, h, \Bhat)(i) = j$, we get
\[ F(\Bhat, g^{-1}, \Bhat_{g})(j) = i'.\]
Since $g$ was chosen arbitrarily,
\[ h(\cbar) \forces_{\B^*} F(\Bhat, \dot{g}^{-1}, \Bhat_{\dot{g}})(F(h)(i)) = i'.\qedhere\]
\end{proof}

\begin{remark}\label{rmk:comm}
In the proof of the previous lemma, we saw that the following diagram commutes:
\[
\xymatrix{
F(\Btilde)\ar[d]_{F(h)}  \ar[r]^{\Fra^{\Btilde}}     & \Dom{\Btilde}{\A}\ar[d]^{\tilde{h}}   \\
F(\Bhat)   \ar[r]_{\Fra^{\Bhat}}          & \Dom{\Bhat}{\A} 
}\]
We will use this fact later.
\end{remark}

The last thing we need to verify is the complexity claim.

\begin{proposition}
For any $\Delta^0_\alpha$ functor $F\colon \Iso{\B} \to \Iso{\A}$ there is a $\Dec_\alpha$ interpretation, $\I$, of $\A$ within $\B$, such that $F$ is naturally isomorphic to the functor $F_\I$ associated to $\I$.
Furthermore, the isomorphism between $F$ and $F_\I$ can be taken to be $\Delta^0_\alpha$.
\end{proposition}
\begin{proof}
That $\I$ is a $\Dec_\alpha$ interpretation follows immediately from Lemma \ref{lem:definability}, the definition of the interpretation, and our remark that if $F$ is a $\Delta^0_\alpha$ functor, then the formulas involved in the definition of the interpretation are all $\Dec_\alpha$. That the isomorphism between $F$ and $F_{\I}$ is $\Delta^0_\alpha$ follows from the fact that determining whether
\[ \cbar\forces_{\B^*} F(\Btilde, \dot{g}^{-1}, \Btilde_{\dot{g}})(i) = j \]
is $\Delta^0_\alpha(\Btilde)$ uniformly in $\Btilde$.
\end{proof}

\subsection{Bi-interpretations}

\begin{proof}[Proof of Theorem \ref{thm:eq to bi}]
Let $F \colon \Iso{\B} \to \Iso{\A}$ and $G \colon \Iso{\A} \to \Iso{\B}$ be a Borel adjoint equivalence of categories, as in the statement of the theorem, with $\eta \colon \id_{\Iso{\B}} \to GF$ and $\epsilon_{\Atilde} \colon \id_{\Iso{\A}} \to FG$. Assume that $\A = F(\B)$.

Let $\I$ and $\J$ be the interpretations using the method described earlier. Recall that just before Lemma \ref{lem:nat-iso} we defined an operator $\Fra$ which, for each $\Btilde$, gives an isomorphism $\Fra^{\Btilde} \colon F(\Btilde) \to \Dom{\Btilde}{\A}$. We get such an operator for each of $F$ and $G$, denoting them by $\Fra^{\Btilde} \colon F(\Btilde) \to \Dom{\Btilde}{\A}$ and $\Gra^{\Atilde} \colon G(\Atilde) \to \Dom{\Atilde}{\B}$.

Consider the isomorphism
\[ \tilde{\Fra}^{\Bhat}\circ\Gra^{F(\Bhat)}\circ\eta_{\Bhat} \colon \Bhat \to \Do_{\B}^{\Do_{\A}^{\Bhat}}.\]
Let $h \colon \Bhat \to \Bti$ be an isomorphism. Then we get maps
\[ \xymatrix{\Bhat\ar[d]_{h}\ar[r]^{\eta_{\Bhat}}&G(F(\Bhat))\ar[d]^{G(F(h))}\ar[r]^{\Gra^{F(\Bhat)}}&\Do_{\B}^{F(\Bhat)}\ar[d]^{\tilde{F}(h)}\ar[r]^{\tilde{\Fra}^{\Bhat}}&\Do_{\B}^{\Do_{\A}^{\Bhat}}\ar[d]_{\tilde{\tilde{h}}}\\\Bti\ar[r]_{\eta_{\Bti}}&G(F(\Bti))\ar[r]_{\Gra^{F(\Bti)}}&\Do_{\B}^{F(\Bti)}\ar[r]_{\tilde{\Fra}^{\Bhat}}&\Do_{\B}^{\Do_{\A}^{\Bti}}}\]
The first square commutes because $\eta$ is a natural isomorphism $\id_{\Iso{\B}} \to GF$, and the remaining two squares commute by Remark \ref{rmk:comm}.

First, take $\Btilde = \Bhat = \B$ and $h$ an automorphism of $\B$. We see from the fact that the diagram above commutes that $\tilde{\Fra}^{\B}\circ\Gra^{F(\B)}\circ\eta_{\B}$ is invariant under automorphisms of $\B$, and so it is $\L_{\omega_1 \omega}$-definable.

Now we claim that if $F$ is $\Delta^0_\alpha$ (or $\bfDelta^0_\alpha$), then $\tilde{\Fra}^{\B}\circ\Gra^{F(\B)}\circ\eta_{\B}$ is relatively intrinsically $\Delta^0_\alpha$ (resp. $\bfDelta^0_\alpha$) and hence definable by a $\Dec_\alpha$ (resp. $\Dei_\alpha$) formula. Consider the commutative diagram above, with $\Btilde = \B$ and $\Bhat$ some other copy of $\B$, with an isomorphism $h \colon \B \to \Bhat$. We see that
\[ \tilde{\Fra}^{\Bhat}\circ\Gra^{F(\Bhat)}\circ\eta_{\Bhat} \colon \Bhat \to \Do_{\B}^{\Do_{\A}^{\Bhat}} \]
is defined within $\Bhat$ by the same formula which defines $\tilde{\Fra}^{\B}\circ\Gra^{F(\B)}\circ\eta_{\B}$ in $\B$. Moreover, since $F$, $G$, and $\eta$ are $\Delta^0_\alpha$ (resp. $\bfDelta^0_\alpha$) operators, $\tilde{\Fra}^{\Bhat}\circ\Gra^{F(\Bhat)}\circ\eta_{\Bhat}$ is $\Delta^0_\alpha$ (resp. $\bfDelta^0_\alpha$) in $\Bhat$. Thus $\tilde{\Fra}^{\B}\circ\Gra^{F(\B)}\circ\eta_{\B}$ is relatively intrinsically $\Delta^0_\alpha$ (resp. $\bfDelta^0_\alpha$). A similar argument works for $\tilde{\Gra}^{\A}\circ\Fra^{G(\A)}\circ\epsilon_{\A}$.

Define $g^\B_\A = {\Fra}^{\B} \colon \A \to \Do_{\A}^{\B}$ and $g^\A_\B = \Gra^{F(\B)} \circ\eta_{\B} \colon \B \to \Do_{\B}^{\A}$. We know that these are isomorphisms. The maps $g^\B_\A$ and $g_\A^\B$ go in the opposite direction as the maps from Definition \ref{defn:inf-biinterpretable}. Letting $f^\B_\A$ and $f_\A^\B$ be their inverses, we get the maps required for a bi-interpretation. We just have to show that the compositions of these maps are the $\L_{\omega_1 \omega}$-definable isomorphisms from the previous paragraph.

We have
\[ \tilde{g}^\B_\A \circ g^\A_\B = \tilde{\Fra}^{\B}\circ\Gra^{F(\B)}\circ\eta_{\B} \colon \B \to \Do_{\B}^{\Do_{\A}^{\B}}. \]
Also, by Remark \ref{rmk:comm} (With $h = \eta_{\B}$), and the fact that $F(\eta_\B) = \epsilon_{F(\B)}$,
\[ \tilde{\eta}_{\B} \circ \Fra^{\B} = \Fra^{G(F(\B))} \circ F(\eta^\B) = \Fra^{G(F(\B))} \circ \epsilon_{F(\B)}. \]
Then, using the fact that $\A = F(\B)$,
\[ \tilde{\Gra}^{\A}\circ\Fra^{G(\A)}\circ\epsilon_{\A} = \tilde{\Gra}^{\A} \circ \tilde{\eta}_{\B} \circ \Fra^{\B} = \tilde{g}^\A_\B \circ g^\B_\A.\] 
\end{proof}

\begin{proof}[Proof of Theorem \ref{thm:bi to eq}]
Let $\I$ and $\J$ be as in the statement of the theorem: $\I$ is an interpretation of $\A$ inside of $\B$ and $\J$ is an interpretation of $\B$ inside of $\A$. From these bi-interpretations we get functors $F = F_\I \colon \Iso{\B} \to \Iso{\A}$ and $G = F_\J \colon \Iso{\A} \to \Iso{\B}$. These functors were defined so that, for $\Bti \in \Iso{\B}$ and $\Ati \in \Iso{\A}$, there are isomorphisms $\tau^\Bti \colon F(\Bti) \to \Do_{\A}^{\Bti}$ and $\rho^\Ati \colon G(\Ati) \to \Do_{\B}^{\Ati}$. Moreover, given an isomorphism $h \colon \Bti \to \Bhat$, $F(h) = {\tau^\Bhat}^{-1} \circ \tilde{h} \circ {\tau^\Bti}$ and given $h \colon \Ati \to \Ahat$, $G(h) =  {\rho^\Ahat}^{-1} \circ \tilde{h} \circ {\rho^\Ati}$.

Recall that $f_\B^\A \circ \tilde{f}_\A^\B$ is $\L_{\omega_1 \omega}$-definable as a subset of $\B^{< \omega}$. In any copy $\Bti$ of $\B$, let $\varphi^\Bti \colon \Do_\B^{\Do_\A^{\Bti}} \to \Bti$ be defined by this formula.  Similarly, let $\psi^\Ati \colon \Do_\A^{\Do_\B^{\Ati}} \to \Ati$ in $\Ati$ be defined by the same formula which defines $f_\A^\B \circ \tilde{f}_\B^\A$ in $\A$. 

Essentially what we want to do is to identify $\Do_\B^{\Do_\A^{\Bti}}$ with $G(F(\Bti))$ and $\Do_\A^{\Do_\B^{\Ati}}$ with $F(G(\Ati))$ and use $\varphi$ and $\psi$ as our natural isomorphisms. We make these identifications using $\tau$ and $\rho$.

Define
\[ \eta_\Bti = \varphi^{\Bti} \circ \tilde{\tau}^{\Bti} \circ \rho^{F(\Bti)} \colon G(F(\Bti)) \to \Bti \]
and
\[ \epsilon_\Ati = \psi^{\Ati} \circ \tilde{\rho}^{\Ati} \circ \tau^{G(\Ati)} \colon F(G(\Ati)) \to \Ati. \]
We begin by showing that $\eta$ is a natural isomorphism between $GF$ and $\id_{\Iso{\B}}$. (Note that it is more convenient here to have $\eta$ and $\epsilon$ mapping in the opposite direction as in Definition \ref{def: adj equiv cat}.)
Let $h \colon \Bti \to \Bhat$ be an isomorphism. Then since $\varphi$ is $\L_{\omega_1 \omega}$-definable,
\begin{align*}
h \circ \eta_\Bti &= h \circ \varphi^{\Bti} \circ \tilde{\tau}^{\Bti} \circ \rho^{F(\Bti)} \\
&= \varphi^{\Bhat} \circ \tilde{\tilde{h}} \circ \tilde{\tau}^{\Bti} \circ \rho^{F(\Bti)}. 
\end{align*}
By definition of $F$ and $G$,
\begin{align*}
\varphi^{\Bhat} \circ \tilde{\tilde{h}} \circ \tilde{\tau}^{\Bti} \circ \rho^{F(\Bti)} &= \varphi^{\Bhat} \circ \tilde{\tau}^{\Bhat} \circ \widetilde{F(h)} \circ \rho^{F(\Bti)} \\
&= \varphi^{\Bhat} \circ \tilde{\tau}^{\Bhat} \circ \rho^{F(\Bhat)} \circ G(F(h)) \\ 
&= \eta_\Bhat \circ G(F(h)).
\end{align*}
Similarly, $\epsilon$ is a natural isomorphism between $FG$ and $\id_{\Iso{\A}}$.
Thus we have shown that $F$, $G$, $\eta$, and $\epsilon$ give an equivalence of categories.

Now we must show that this is equivalence is an adjoint equivalence by showing that given $\Bti$ and $\Ati$, $F(\eta_{\Bti}) = \epsilon_{F(\Bti)}$ and $G(\epsilon_{\Ati}) = \eta_{G(\Ati)}$. To begin, we prove two claims which give identities of compositions of isomorphisms.

\begin{claim}
$\tilde{\varphi}^{\B} \circ \tilde{\tilde{\tau}}^{\B} = \tau^{\B} \circ \psi^{F(\B)}$
\end{claim}
\begin{proof}
Let $h \colon \A \to F(\B)$ be an isomorphism. Then since $h^{-1} \circ  \psi^{F(\B)} \circ \tilde{\tilde{h}} = \psi^{\A}$, we just need to show that
\[ \tilde{f}_\B^\A \circ \tilde{\tilde{f}}_\A^\B \circ \tilde{\tilde{\tau}}^{\B} \circ \tilde{\tilde{h}} = \tau^{\B} \circ h \circ f_\A^\B \circ \tilde{f}_\B^\A. \]
Consider the isomorphisms $f_\A^\B \colon \Do_\A^\B \to \A$ and $\tau^{\B} \circ h \colon \A \to \Do_\B^\A$. Let $\alpha = f_\A^\B \circ \tau^{\B} \circ h$. Then $\alpha$ is an automorphism of $\A$. Since $f_\A^\B \circ \tilde{f}_\B^\A$ is definable, we have
\begin{align*}
\tilde{f}_\B^\A \circ \tilde{\tilde{f}}_\A^\B \circ {\tilde{\tilde{f}}_\A^\B}^{-1} \circ \tilde{\tilde{\alpha}} &= {f_\A^\B}^{-1} \circ f_\A^\B \circ \tilde{f}_\B^\A \circ \tilde{\tilde{\alpha}} \\
\tilde{f}_\B^\A \circ \tilde{\tilde{f}}_\A^\B \circ {\tilde{\tilde{f}}_\A^\B}^{-1} \circ \tilde{\tilde{\alpha}} &= {f_\A^\B}^{-1} \circ \alpha \circ f_\A^\B \circ \tilde{f}_\B^\A \\
\tilde{f}_\B^\A \circ \tilde{\tilde{f}}_\A^\B \circ \tilde{\tilde{\tau}}^{\B} \circ \tilde{\tilde{h}} &= \tau^{\B} \circ h \circ f_\A^\B \circ \tilde{f}_\B^\A
\end{align*}
as desired.
\end{proof}

The next claim replaces $\B$ in the claim above by an arbitrary copy $\Btilde$ of $\B$.

\begin{claim}
$\tilde{\varphi}^{\Bti} \circ \tilde{\tilde{\tau}}^{\Bti} = \tau^{\Bti} \circ \psi^{F(\Bti)}$.
\end{claim}
\begin{proof}
Let $h \colon \Bti \to \B$ be an isomorphism. By definition of $F$, we have
\[ \tilde{h} \circ \tau^{\Bti}  = \tau^{\B} \circ F(h). \]
Since $\varphi$ and $\psi$ are given by $\L_{\omega_1 \omega}$ definitions,
\[ h \circ \varphi^{\Bti} = \varphi^{\B} \circ \tilde{\tilde{h}} \text{ and } F(h) \circ \psi^{F(\Bti)}  = \psi^{F(\B)} \circ \widetilde{\widetilde{F(h)}}. \]
From the previous claim, we get
\[ \tilde{\varphi}^{\B} \circ \tilde{\tilde{\tau}}^{\B} = \tau^{\B} \circ \psi^{F(\B)}. \]
Composing both sides with $\tilde{\tilde{F(h)}}$ on the right, we get
\begin{align*}
\tilde{\varphi}^{\B} \circ \tilde{\tilde{\tau}}^{\B} \circ \tilde{\tilde{F(h)}} = \tau^{\B} \circ \psi^{F(\B)} \circ \widetilde{\widetilde{F(h)}} \\
\tilde{\varphi}^{\B} \circ \tilde{\tilde{\tilde{h}}} \circ \tilde{\tilde{\tau}}^{\Bti} = \tau^{\B} \circ F(h) \circ \psi^{F(\Bti)} \\
\tilde{h} \circ \tilde{\varphi}^{\Bti} \circ \tilde{\tilde{\tau}}^{\Bti} = \tilde{h} \circ  \tau^{\Bti} \circ \psi^{F(\Bti)}
\end{align*}
Applying $\tilde{h}^{-1}$ to both sides, we complete the claim.
\end{proof}

To see that $F(\eta_{\Bti}) = \epsilon_{F(\Bti)}$, note that
\[ F(\eta_\Bti) = {\tau^{\Bti}}^{-1} \circ \tilde{\varphi}^{\Bti} \circ \tilde{\tilde{\tau}}^{\Bti} \circ \tilde{\rho}^{F(\Bti)} \circ \tau^{G(F(\Bti))} \colon F(G(F(\Bti))) \to F(\Bti) \]
and
\[ \epsilon_{F(\Bti)} = \psi^{F(\Bti)} \circ \tilde{\rho}^{F(\Bti)} \circ \tau^{G(F(\Bti))} \colon F(G(F(\Bti))) \to F(\Bti). \]
Then it follows from the previous claim that these are equal. Similarly, $G(\epsilon_{\Ati}) = \eta_{G(\Ati)}$.
\end{proof}

\section{Indiscernibles}

In this section we prove Theorem \ref{thm:indiscernibles}, which says that, for a structure $\A$, there is a continuous homomorphism from $\Aut(\A)$ onto $S_\infty$ if and only if $\A$ has an infinite definable set of absolutely indiscernible definable equivalence classes.

\begin{proof}[Proof of Theorem \ref{thm:indiscernibles}]
The direction (2)$\Rightarrow$(1) is easy to see. For the other direction, suppose that there is a continuous homomorphism $H$ from $\Aut(\A)$ onto $S_\infty$. Let $\B$ be the trivial structure with a countable domain and no relations; then $\Aut(\B) = S_\infty$. By Theorem \ref{homo-to-interp}, there is an interpretation $\I$ of $\B$ in $\A$ such that $H = G_\I\upto \Aut(\A)$.

Let $D = \Do_\B^\A \subseteq \A^{< \omega}$, and let $E$ be the relation $\sim$. Let $h$ be a permutation of the $E$-equivalence classes. Then $h$ induces an automorphism $f_\B^\A \circ h \circ {f_\B^\A }^{-1}$ of $\B$. Then, since $H$ is onto, there is an automorphism $g$ of $\A$ with $H(g) = f_\B^\A \circ h \circ {f_\B^\A }^{-1}$. But then $G_\I(g) = H(g) = f_\B^\A \circ h \circ {f_\B^\A }^{-1}$, and so, by definition of $G_\I$, $g$ extends $h$.

Above, we chose $D \subseteq \A^{< \omega}$; we need to choose $D \subseteq \A^{n}$ for some $n$. It suffices to show that for some $n$, $D' = D \cap \A^{n}$ and $E' = E \cap (D' \times D')$ have infinitely many equivalence classes. Let $n$ be such that $D \cap \A^{n}$ is non-empty:  say it contains some element $\bar{a}$. Let $x \in \B$ be $f_\B^\A(\bar{a})$. Let $y_1,y_2,\ldots$ be infinitely many elements of $\B$ distinct from $x$, and let $h_1,h_2,\ldots$ be automorphisms of $\B$ such that $h_i(x) = y_i$. Then since $H$ is onto, there are automorphisms $g_i$ of $\A$ with $H(g_i) = h_i$. Then $h_i = G_\I(g_i) = f_\B^\A \circ g_i \circ {f_\B^\A}^{-1}$. Since $h_i(x) = y_i$ and $f_\B^\A(\bar{a}) = x$, $f_\B^\A \circ g_i(\bar{a}) = y_i$. Thus $g_i(\bar{a})$ must be in a different $E$-equivalence class from $\bar{a}$, and also from $g_j(\bar{a})$ for $i \neq j$; but since $\bar{a} \in \A^{n}$, $g_i(\bar{a}) \in \A^n$. Thus there are infinitely many $E$-equivalence classes in $D \cap \A^{n}$.
\end{proof}

A similar argument proves the following theorem.

\begin{theorem}\label{thm: order indiscernibles}
Let $\A$ be a countable structure. 
The following are equivalent:
\begin{enumerate}
\item There is a continuous homomorphism from $\Aut(\A)$ onto $\Aut(\mathbb{Q},<)$.
\item There is an $n$, an $\L_{\om_1 \om}$-definable $D\subset A^n$, an $\L_{\om_1 \om}$-definable equivalence relation $E\subset D^2$ with infinitely many equivalence classes, and an $\L_{\om_1 \om}$-definable order, such that the $E$-equivalence classes are order indiscernible, in the sense that each order-preserving permutation of the $E$-equivalence classes extends to an automorphism of $\A$.
\end{enumerate}
\end{theorem}

By considering isomorphisms, we also get:

\begin{theorem}
\label{thm:indiscernibleiso}
Let $\A$ be a countable structure. 
The following are equivalent:
\begin{enumerate}
\item There is a continuous isomorphism between $\Aut(\A)$ and $S_\infty$.
\item There is an $n$, an $\L_{\om_1 \om}$-definable $D\subset A^n$, and an $\L_{\om_1 \om}$-definable equivalence relation $E\subset D^2$ with infinitely many equivalence classes and such that the $E$-equivalence classes are absolutely indiscernible, and every other element is definable from this set. In other words, if we add relations naming each of these equivalence classes, then every element of the structure is $\L_{\om_1 \om}$-definable.
\end{enumerate}
\end{theorem}

\begin{theorem}
\label{thm:orderindiscernibleiso}
Let $\A$ be a countable structure. 
The following are equivalent:
\begin{enumerate}
\item There is a continuous isomorphism between $\Aut(\A)$ and $\Aut(\mathbb{Q},<)$.
\item There is an $n$, an $\L_{\om_1 \om}$-definable $D\subset A^n$, an $\L_{\om_1 \om}$-definable equivalence relation $E\subset D^2$ with infinitely many equivalence classes, and an $\L_{\om_1 \om}$-definable order, such that the $E$-equivalence classes are order indiscernible, and every other element is definable from this set.
\end{enumerate}
\end{theorem}

\bibliographystyle{alpha}
\bibliography{References}

\end{document}